\documentclass[a4paper, 10pt]{article}
\usepackage{amsmath}
\usepackage{amsfonts}
\usepackage{amssymb}
\usepackage[english]{babel}
\usepackage{graphicx}
\usepackage{amsthm}
\usepackage{pgf,tikz}
\usetikzlibrary{arrows}
\usepackage[title]{appendix}

\newtheorem{thm}{Theorem}[section]

\newtheorem{lem}[thm]{Lemma}

\newtheorem{ex}[thm]{Example}
\theoremstyle{definition}

\theoremstyle{remark}
\newtheorem{rem}{Remark}[section]
\usetikzlibrary{snakes}
\usepackage{color}

\numberwithin{equation}{section}

\newcommand{\HH}{\mathcal{H}}
\newcommand{\EE}{\mathcal{E}}
\newcommand{\TT}{\mathcal{T}}
\newcommand{\AAA}{\mathbb{A}}

\newcommand{\Z}{\mathbb{Z}}
\newcommand{\Q}{\mathbb{Q}}

\newcommand{\ol}{\overline}
\date{}
\begin{document}
\title{On the full automorphism group of a Hamiltonian cycle system of odd order
%\thanks{Work performed under the auspicies of the G.N.S.A.G.A. of the
%C.N.R. (National Research Council) of Italy and supported by M.I.U.R. project "Disegni combinatorici, grafi e loro applicazioni, PRIN 2008".}
}

\author{
Marco Buratti
\thanks{Dipartimento di Matematica e Informatica, Universit\`a degli Studi di Perugia, via Vanvitelli 1 - 06123 Italy,
email: buratti@dmi.unipg.it}\quad\quad
Graham J. Lovegrove
\thanks{Department of Mathematics and Statistics, The Open University, Walton Hall, Milton Keynes MK7 6AA, U.K.,
email: graham.lovegrove@virgin.net}\quad\quad
Tommaso Traetta
\thanks{Dipartimento di Matematica e Informatica, Universit\`a degli Studi di Perugia, via Vanvitelli 1 - 06123 Italy,
email: traetta@dmi.unipg.it}}

\maketitle
\begin{abstract}
\noindent 
It is shown that a necessary condition for an abstract group $G$ to be the full automorphism group of a Hamiltonian cycle system is that $G$ has odd order or it is either binary, or the affine linear group AGL($1,p$) with $p$ prime. We show that this condition is also sufficient except possibly for the class of non--solvable binary groups.
\end{abstract}


\eject
\section{Introduction}
Denote as usual with $K_v$ the complete graph on $v$ vertices.
A Hamiltonian cycle system of order $v$ (briefly, an HCS$(v)$) is a set
of Hamiltonian cycles of $K_v$ whose edges partition the edge-set of $K_v$. It is very well known \cite{L} that an HCS($v$) exists if and only if $v$ is odd and $v \geq 3$.

Two HCSs are {\it isomorphic} if there exists a bijection ({\it isomorphism}) between their set of vertices
turning one into the other.  An automorphism of a HCS$(v)$ is an isomorphism of it with itself, i.e., a permutation of the vertices of $K_v$ leaving it invariant.

HCSs possessing a non-trivial automorphism group have drawn a certain attention (see \cite{HCShort} for a short recent survey on this topic). Detailed results can be found in: \cite{BD,JM} for the {\it cyclic groups}; \cite{BM} for the {\it dihedral groups}; \cite{BBM} for the {\it doubly transitive groups};
\cite{B} for the {\it regular HCS}; \cite{AKN, BS} for the {\it symmetric HCS}; \cite{BM2} for those being both cyclic and symmetric; \cite{BOP, BRT, BBRT} for the {\it $1$-rotational HCS} and {\it $2$--pyramidal HCS}.

Given a particular class of combinatorial designs, to establish whether any abstract
finite group is the full automorphism group of an element
in the class is in general a quite hard task.
Some results in this direction can be found in \cite{Me3}
for the class of Steiner triple and quadruple
systems, in \cite{Me1} for the class of finite projective planes, in \cite{BMR}
for the class of non--Hamiltonian $2-$factorizations of the complete graph, and very recently in \cite{GGL, Love08, Love14} for the class of cycle systems.

This paper deals with the following problem:
\begin{quote}
{\it Determining the class $\mathcal{G}$ of finite groups that can be seen as the full automorphism group of a Hamiltonian cycle system of odd order. }
\end{quote}
As a matter of fact, 
some partial answers are known. In \cite{BRT} it is proven that any symmetrically sequenceable group lies in $\mathcal{G}$ (see \cite{BBRT} for HCSs of even order). In particular, any solvable binary group (i.e. with a unique element of order $2$) except for the quaternion group $\Q_8$, is symmetrically sequenceable \cite{AI}, hence it can be seen as the full automorphism group of an HCS. In \cite{BBM} it is shown that 
the affine linear group $AGL(1,p)$, $p$ prime, is the full automorphism group of the unique doubly transitive HCS$(p)$.

Here we prove that any finite group $G$ of odd order is the full automorphism group of an HCS. This result will be achieved in Section \ref{fullaut} by means of a new doubling construction described in Section \ref{doubling}. On the other hand, in Section \ref{fullaut} we also prove that if $G \in \mathcal{G}$ has even order, then $G$ is necessarily binary or the affine linear group AGL($1,p$) with $p$ prime; still, we show that $\Q_8$ lies in $\mathcal{G}$. We obtain, in this way, the major result of this paper:
\begin{thm}\label{them0}
   If a finite group $G$ is the full automorphism group of a Hamiltonian cycle system of odd order then $G$ has odd order or it is either binary, or the affine linear group AGL($1,p$) with $p$ prime. The converse is true except possibly in the case of G binary non-solvable.
\end{thm}
We therefore leave open the problem only for non--solvable binary groups.

\section{A new doubling construction}\label{doubling}
We describe a new doubling construction that will allow us to constuct an HCS($4n+1$) starting from three HCS($2n+1$) not necessarily distinct.\\

For any even integer $n \ge 1$, take three HCS($2n+1$), say $\HH_1, \HH_2, \HH_3$, on the set $\{\infty\} \ \cup  \ [2n]$, where $[2n]=\{1,2,\ldots,2n\}$. Denote by $A_i$, $B_i$, $C_i$, for $i\in [n]$, the cycles composing $\HH_1$, $\HH_2$, $\HH_3$, respectively, let 
\[A_i=(\infty, \alpha_{i,1}, \ldots, \alpha_{i,2n})\ \quad \
  B_i=(\infty, \beta_{i,1},  \ldots, \beta_{i,2n}) \ \quad \  
  C_i=(\infty, \gamma_{i,1},  \ldots, \gamma_{i,2n}). 
\]
We need these HCS to satisfy the following property:
\begin{equation}\label{property}
  \alpha_{i1} = \beta_{i1} = \gamma_{i1}, \quad \text{and} \quad  
  \alpha_{i,2n} = \beta_{i,2n} = \gamma_{i,2n}, \quad \text{for $i\in [n]$}.
\end{equation}
We construct an HCS($4n+1$) $\mathcal{T}$ on the set $\{\infty\} \ \cup\ 
([2n]\times \{1,-1\})$. %For convenience, given $\alpha_1,\alpha_2,\ldots,\alpha_m$ in $[m]$ we will henceforth denote the point $(\alpha_i,j)$ by $\alpha_{i,j}$. 
For convenience, if $z=(x,y)\in [2n]\times \{1,-1\}$, then the point $(x,-y)$ will be denoted by 
$z'$. 

Let $\mathcal{T}=\{T_{i1}, T_{i2} \;|\; i \in [n]\}$ be the set of $2n$ cycles of length $4n+1$ and vertex-set
$\{\infty\} \ \cup\ ([2n]\times \{1,-1\})$
obtained from the cycles of $\HH_1, \HH_2$, $\HH_3$ as follows: 
set $a_{ij} = (\alpha_{ij},1)$, $b_{ij} = (\beta_{ij},1)$, 
and $c_{ij} = (\gamma_{ij},1)$, for $i\in [n]$ and $j \in [2n]$ and 
define the cycles $T_{i1}$, $T_{i2}$, of {\it first} and {\it second} type respectively, as follows:
\begin{align*}
T_{i,1} = &(\infty,a_{i,1}, a_{i,2}, \ldots, a_{i,2n}, b'_{i,2n}, b'_{i,2n-1},\ldots,b'_{i1}), \\
T_{i,2} = &(\infty, c_{i,2n},c'_{i,2n-1},c_{i,2n-2}, c'_{i,2n-3}, 
\ldots,c'_{i1}, c_{i,1}, c'_{i,2},c_{i,3},c'_{i,4},\ldots,c'_{i,2n}).
\end{align*}

\begin{rem}
  Note that, by construction, the neighbors of $\infty$ and the middle edge of any of the cycles of $\TT$ are both pairs of the form $(z,z')$ for $z\in [2n]\times \{-1,1\}$. Also,  if $T_1=(\infty, z, \ldots, w,w',\ldots, z')$ is a cycle of $\TT$, then there is a cycle 
  $T_2=(\infty, w,\ldots, z', z,\ldots, w')$ of $\TT$ of alternate type.
\end{rem}

We first show that our doubling construction yields a Hamiltonian cycle system.

\begin{lem}
\label{lem1}
$\mathcal{T}$ is an HCS($4n+1$).
\end{lem}
\begin{proof}
Every unordered pair of form $(x_1,x_2)$ or $(\infty,x_1)$, $x_1,x_2\in [2n], 
x_1\ne x_2$ is contained in a unique cycle of each of $\mathcal{H}_1, $$\mathcal{H}_2, $$\mathcal{H}_3$. 
The first type of cycle above contains all pairs of form $((x_1,y),(x_2,y))$, $x_1, x_2 \in [2n]$, $x_1 \neq x_2$, $y=-1,1$, and the second cycle type contains all pairs of form $((x_1,1),(x_2,-1))$, $x_1, x_2 \in [2n]$, 
$x_1 \neq x_2$. We are left to show that $\TT$ contains the following edges: 
$(\infty, z)$ and $(z,z')$ for $z \in [2n]\times \{-1,1\}$.

In view of Property \eqref{property}, we have that $a_{ij}= b_{ij} = c_{ij}$ for $j=1,2n$. Therefore, the middle edges of the cycles in $\TT$ are exactly  $(a_{i1}, a'_{i1})$
and $(a_{i,2n},a'_{i,2n})$ for $i\in [n]$. Considering that $\{a_{i1}, a_{i,2n}\;|\;i \in [n]\}=[2n]\times \{1\}$, we conclude that $\mathcal{T}$ covers the edges  
$(z,z')$ for $z \in [2n]\times\{-1,1\}$.

Finally, the edges incident with $\infty$ and covered by $\TT$ are the following: 
$(\infty, a_{i1}),$ $(\infty, b'_{i1}), (\infty, c_{i,2n}), (\infty, c'_{i,2n})$, for 
$i\in [n]$. With a reasoning similar to the former we easily see 
that all edges $(\infty, z)$ with $z \in [2n]\times \{-1,1\}$ are covered by $\TT$. 
\end{proof}

\begin{ex} \label{ex1}
Here we show how to construct an HCS(13) by applying the doubling construction to three HCSs or order 7. 
Let $G=\langle g\rangle$ be the cyclic group of order $6$ generated by $g$, and let $\HH_1, \HH_2$, and $\HH_3$ denote
the three HCSs of order $7$ defined as follows: 
\begin{enumerate}
  \item $\HH_1 = \{A_1, A_2, A_{3}\}$ with
  $A_1=(\infty, 1, g, g^5, g^2, g^4, g^3)$ and $A_i=A_1\cdot g^{i-1}$ for $i=2,3$,
  \item $\HH_2 = \{B_1, B_2,B_3\}$ with
  $B_1=(\infty, 1, g^4, g^5, g^2, g, g^3)$ 
  and $B_i=B_1\cdot g^{i-1}$ for $i=2,3$,
  \item $\HH_3=\{C_1, C_2, C_3\} $ and $C_i = B_i$ for $i=1,2,3$ (hence, $\HH_3 = \HH_2$),
\end{enumerate}
where $A_1\cdot g^{i-1}$ ($B_1\cdot g^{i-1}$) is the cycle that we obtain by replacing each vertex of $A_1$ ($B_1$) different from $\infty$, say $x$, with $x\cdot g^{i-1}$. 
It is easy to check that $\HH_1$ and $\HH_2$ are HCS($7$); also, property  \ref{property}
is satisfied (see Figure \ref{img2}).

\begin{figure}[htbp]
  \centering
  \includegraphics[width=1.0\textwidth]{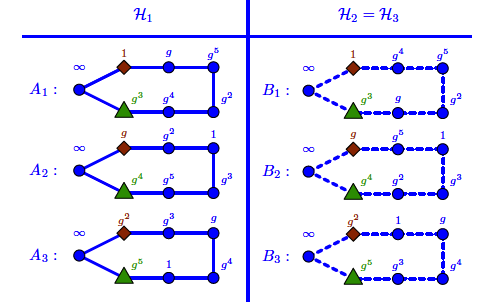}
  \caption{Two HCSs of order $7$.}
  \label{img2}
\end{figure}

The cycles $T_{i1}$ and $T_{i2}$ for $i=1,2,3$ are shown in Figure \ref{img3}.
Each cycle $T_{i1}$ is basically constructed from the paths that we obtain from $A_i$ (continous path) and $B_i$ (dashed path) after removing $\infty$, by joining two of their ends with $\infty$ and joining to each other the other two ends (zigzag edges). 
The construction of each cycle $T_{i2}$ is only based  on $C_i=B_i$. Consider two copies of the path we obtain from  $B_i$ (dashed path) after removing $\infty$ and replace the horizontal edges with the diagonal ones. At the end, we add the zigzag edges.

It is easy to check that the set $\mathcal{T}=\{T_{i1}, T_{i2} \;|\; i=1,2,3\}$ is an HCS($13$).

\begin{figure}[htbp]
 \centering
 \includegraphics[width=1.0\textwidth]{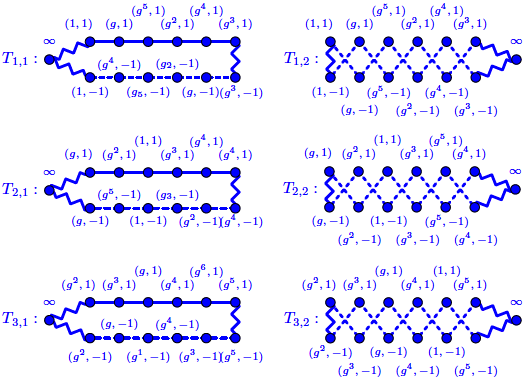}
 \caption{An HCS(13) resulting from the doubling condstruction.}
 \label{img3}
\end{figure}
\end{ex}

\begin{rem} This construction can be used to construct many different HCS from the same base systems by relabelling the vertices of one or more of $\HH_1, \HH_2, \HH_3$ in a way that Property \eqref{property} is still satisfied. 
%The order in which we write the vertices of the cycles in $\HH_1, \HH_2, \HH_3$ plays an important role in our construction. Indeed, imagine to reverse the writing of even one of these cycles; if Property \eqref{property} is still satisfied, then it is easy to see that the HCS($4n+1$) that we get is different from the previous one.
\end{rem}

\begin{lem}
\label{lem2}
All automorphisms of $\mathcal{T}$ fix the $\infty$ point for $n>1$.
\end{lem}
\begin{proof}
Suppose there is an automorphism $\phi$ of $T$ 
which does not fix the $\infty$ point, and let $u=\phi(\infty)$ and
$\phi^{-1}(\infty)=v$. Also denote $\phi(v')$ by $w$. The cycle $T_1$ containing $(\infty,v)$ as an edge has form: $(\infty, v,\ldots, z,z',\ldots,v')$.
This is mapped by $\phi$ to the cycle $\phi(T_1) = 
(u,\infty,\ldots, \phi(z),\phi(z'),\ldots,w)$. 

Now consider the cycle $T_2$ that is of alternate type to $T_1$ and that has the middle edge $(v,v')$, i.e. $T_2=(\infty,z,\ldots,v',v,\ldots,z')$. 
Then $T_2$ is mapped to the cycle 
$\phi(T_2) = (u,\phi(z),\ldots,w,\infty,\ldots,\phi(z'))$. 
The middle edge of $\phi(C_2)$ is $(u,\phi(z))$, hence there exists a cycle where the neighbors of $\infty$ are  $u$ and $\phi(z)$. Since $u$ is adjacent to $\infty$ in $\phi(T_1)$, then $\phi(C_1)=(u, \infty, \phi(z), \phi(z'), w)$ and 
$n=1$.
\end{proof}\

\section{HCSs with a prescribed full automorphism group}\label{fullaut}
In this section we prove that any group of odd order lies in the class $\mathcal{G}$ of finite groups than can be seen as the full automorphism group of an HCS of odd order. After that, we prove that whenever a group 
$G\in \mathcal{G}$ has even order then it is either binary or the affine linear group AGL($1,p$), with $p$ prime; also, we show that the quaternion group $\Q_8$ lies in 
$\mathcal{G}$.  As mentioned in the introduction, 
it is known that any binary solvable group $\neq \Q_8$ \cite{BRT}
and AGL($1,p$) ($p$ prime) \cite{BBM} 
lie in $\mathcal{G}$. Therefore, we leave open the problem of determining whether non--solvable binary groups lie in $\mathcal{G}$ as well.

In order to show that any group of odd order is the full automorphism group of a suitable HCS($4n+1$), we will need some preliminaries on $1$--rotational HCSs.

We will use multiplicative notation to denote any abstract 
group; as usual, the unit will be denoted by $1$. 

An HCS($2n+1$) $\HH$ is $1$-rotational over a group $\Gamma$ of order $2n$ if $\Gamma$ is an automorphism group of $\HH$ acting sharply transitively on all but one vertex. In this case, it is natural to identify the vertex-set with $\{\infty\}\ \cup \ \Gamma$ where $\infty$ is the vertex fixed by any $g \in \Gamma$ and view the action of $\Gamma$ on the vertex-set as the right multiplication where $\infty\cdot g=\infty$ for $g \in \Gamma$. 

It is known from \cite{BR} (as a special case of a more general result on $1$-rotational $2$-factorizations of the complete graph) that $G$ is binary, namely, it has only one element $\lambda$ of order $2$. As usual, we denote by 
$\Lambda(\Gamma)=\{1, \lambda\}$ the subgroup of $\Gamma$ of order $2$.

In the same paper, the authors also prove 
that the existence of a $1$-rotational HCS($2n+1$) $\HH$ is equivalent to  the existence of a cycle $A=(\infty, \alpha_1, \ldots, \alpha_{2n})$ with vertex-set $\{\infty\} \ \cup \ \Gamma$ such that
\begin{equation}\label{2starter}
\text{
%$A\cdot \lambda= A$\;\;\; and \;\;\;$\Delta A$ covers twice all elements of 
%$\Gamma\setminus\{1\}$,
$A\cdot \lambda= A$\;\;\; and \;\;\;
$\{\alpha_{i}\alpha_{i+1}^{-1}, 
\alpha_{i+1}\alpha_i^{-1} \;|\; i\in [n-1]\}=
\Gamma\setminus\{1,\lambda\}$,
}
\end{equation} 
%where $\Delta \Gamma = \{a_{i}a_{i+1}^{-1}, a_{i+1}a_i^{-1}\;|\; i\in [2n-1]\}$.
In this case, $\alpha_{n+1-i}=\alpha_{i}\cdot \lambda$ and $\HH=\{A\cdot x\;|\; x \in X\}$, where $X$ is a complete system of representatives for the cosets of 
$\Lambda(\Gamma)$ in $\Gamma$. This means that $\HH$ is the set of distinct translates of any of its cycles.

We are now ready to prove the following result.

\begin{thm}
\label{them1}
Any group $G$ of odd order $n$ is the full automorphism group of a suitable 
HCS$(4n+1)$.
\end{thm}
\begin{proof} Consider the set $\HH$ of the following $13$--cycles:
\begin{align*}
 & (\infty, 0_0, 1_0, 5_0, 2_0, 4_0, 3_0, 3_1, 1_1, 2_1, 5_1, 4_1, 0_1) \\
 & (\infty, 1_0, 2_0, 0_0, 3_0, 5_0, 4_0, 4_1, 2_1, 3_1, 0_1, 5_1, 1_1) \\
 & (\infty, 2_0, 3_0, 1_0, 4_0, 0_0, 5_0, 5_1, 3_1, 4_1, 1_1, 0_1, 2_1) \\  
 & (0_0, 4_1, 5_0, 2_1, 1_0, 3_1, \infty, 3_0, 1_1, 2_0, 5_1, 4_0, 0_1) \\
 & (1_0, 5_1, 0_0, 3_1, 2_0, 4_1, \infty, 4_0, 2_1, 3_0, 0_1, 5_0, 1_1) \\
 & (2_0, 0_1, 1_0, 4_1, 3_0, 5_1, \infty, 5_0, 3_1, 4_0, 1_1, 0_0, 2_1)    
\end{align*}
It is not difficult to check that $\HH$ is an HCS($13$) with no non-trivial automorphism.

Now, given a group $G$ of odd order $n\geq 3$, we show that there exists an 
HCS($4n+1$) whose full automorphism group is isomorphic to $G$. 
Let $\Z_2=\{1,\lambda\}$, 
and set $\Gamma = \Z_{2}\times G$. 
Since $\Gamma$ is a solvable binary group, it is known from \cite{AI} that there exists an HCS($2n+1$) $\HH$ with vertices $\{\infty\}\ \cup \ \Gamma$ that is $1$-rotational under $\Gamma$. 
Let $G=\{g_1=1, g_2, \ldots, g_n\}$ and for a given cycle 
$C=(\infty, x_1, \ldots, x_{2n})$ 
of $\HH$ set $C_i=C\cdot g_i$ for $i \in [n]$; therefore, 
$\HH=\{C_i\;|\;i\in [n]\}$. We have previously pointed out that $C$ satisfies 
\eqref{2starter}. Therefore, we can write $C=(\infty, x_1, \ldots, x_n, \ol{x_n}, \ldots, \ol{x_{1}})$, where $\ol{x_i}=x_i\cdot \lambda$. Also, there exists $k \in[n-1]\setminus\{1\}$ such that $x_{k}x_{k+1}^{-1}=x_1x_2^{-1}\lambda$ or 
$x_{k}x_{k+1}^{-1}=x_2x_1^{-1}\lambda$. 
We define the ($2n+1$)-cycle $C^*$ as follows:
\begin{align*}
  C^*=(\infty, &x_1, \ol{x}_2, \ldots, \ol{x}_k, x_{k+1}, \ldots, x_n,
  \ol{x}_n, \ldots, \ol{x}_{k+1}, x_k, \ldots, x_2, \ol{x}_1).
\end{align*}
Of course, $C^*\cdot\lambda=C^*$. Also, we have that
\begin{align*}
& \text{$\{x_1\ol{x}_2^{-1}, \ol{x}_2x_1^{-1}\} =
\{x_1x_2^{-1}\lambda, x_2x_1^{-1}\lambda\} = \{x_{k}x_{k+1}^{-1}, x_{k+1}x_{k}^{-1}\}$}\\
& \text{$\{\ol{x}_kx_{k+1}^{-1}, x_{k+1}\ol{x}_k^{-1}\} =
\{x_kx_{k+1}^{-1}\lambda, x_{k+1}x_k^{-1}\lambda\} = 
\{x_{1}x_{2}^{-1}, x_{2}x_{1}^{-1}\}$, and}\\
& \text{$\{\ol{x}_i\ol{x}_{i+1}^{-1} \;|\; i \in [n-1]\setminus\{1,k\}\} = 
\{x_ix_{i+1}^{-1} \;|\; i \in [n-1]\setminus\{1,k\}\}$.}
\end{align*}
Therefore, $C^*$, as well as $C$, satisfies both conditions in \eqref{2starter}.
It follows that $\HH^*=\{C_1^*, \ldots, C_n^*\}$, with $C_{i}^* = C^*\cdot g_i$, (namely, the $G$-orbit $\HH^*$ of $C^*$) 
is an HCS$(2n+1)$.

Now we apply the doubling construction defined above with $\HH_1=\HH^*$ and 
$\HH_{2}=\HH_3=\HH$ and let $\TT$ denote the resulting HCS$(4n+1)$ with vertex set $\{\infty\}\cup \Gamma \times \{1,-1\}$. Note that
Property \ref{property} is satisfied, as
the vertices adjacent with $\infty$ in $C$ and $C^*$ coincide.
The starter cycles $T_{1}$, $T_{2}$ of $\TT$ have the following form: 
\begin{align*}
T_{1} = (\infty, &a_{1}=b_{1}, a_{2},\ldots,a_{2n-1}, a_{2n}=b_{2n}, b'_{2n}, \ldots, b'_{1}),\\
T_{2} = (\infty, &b_{2n}, b'_{2n-1}, \ldots, b_{2}, b'_{1},
b_{1}, b'_{2}, \ldots, b_{2n-1}, b'_{2n}),
\end{align*}
where $b_{j}=(x_{j},1)$, $b'_{j}=(x_{j},-1)$, $a_{j}=(\alpha_{j},1)$, and $(\alpha_1, \ldots, \alpha_{2n})$ is the ordered sequence ($x_1, \ol{x}_2, \ol{x}_3, \ldots$) of the vertices of $C^*$. 

The cycles of $\TT$ are generated from $\TT_1$, $\TT_2$ by the action $(x,i)\mapsto (xg,i)$, $x\in \Gamma$, $i=1,-1$, $g\in G$.

For any $g\in G$, this action is an automorphism of $\TT$, and we denote the group of such automorphisms by $\tau_G$. 

We are going to show that $\tau_{G}=Aut(\TT)$. 
Note that $b_{2n}=b_{1}\cdot(\lambda,1)$ and that 
$b_{1}, b'_{1}, b_{2n},b'_{2n}$ form a complete system of representatives of the $G$-orbit on $\Gamma \times \{1,-1\}$. Therefore, given an automorphism 
$\varphi$ of $\TT$, then there exists $z \in G$ such that 
$\phi=\tau_{z}\varphi$ maps $b_{1}$ to one of $b_{1}, b'_{1}, b_{2n},b'_{2n}$.
It is enough to prove that it is always $b_{1}$. It will follow that 
$\phi$ is the identity, because, since the $\infty$ point is always fixed, all the vertices of $T_{1}$ are fixed. Hence, $\varphi=\tau_{z}^{-1} \in \tau_{G}$. Since $\tau_{G}$ is isomorphic to $G$ we get the assertion. 

Suppose that $\phi(b_{1})=b'_{1}$. It follows that the edge $(\infty, b'_{1})$ of $T_{1}$ is also an edge of $\phi(T_{1})$, therefore $T_{1}=\phi(T_{1})$.
In other words, $\phi$ is the reflection of $T_{1}$ in the axis through $\infty$. In particular, $\phi$  swaps $b_{1}$ and $b'_{1}$, and also $a_{2}$ and $b'_{2}$. This means that $\phi$ fixes the edge $(b_{1},b'_{1})$ and hence it fixes the cycle $T_{2}$ containing this edge. Therefore, $\phi$ is the reflection of $T_{2}$ in the axis through the edge $(b_{1},b'_{1})$. It follows that $\phi$  swaps $b_{2}$ and $b'_{2}$, but this contradicts the previous conclusion as $b_{2}=(x_2,1) \neq (\ol{x}_2,1)=a_{2}$.

Now suppose that $\phi(b_{1})=b_{2n}$. Then $\phi$ maps $T_{1}$ to $T_{2}$. Therefore, $\phi(a_{2n-1})=b_{2}=(x_{2},1)= a_{2n-1}$, that is,  
$\phi$ fixes $a_{2n-1}$. However this
implies that $\phi$ is the identity, since then the edge $(\infty,a_{2n-1})$ and all the points in the cycle that contains it are also fixed. 
A similar argument applies to the possibility that $\phi(b_{1})=b'_{2n}$. 
Thus $\phi$ fixes $b_{1}$.
\end{proof}

We point out to the reader that the HCS($13$) $\mathcal{T}$ of Example \ref{ex1} has been constructed following the proof of Theorem \ref{them1}. 
In fact, $\HH_1$ and $\HH_2$ are two $1$-rotational HCS($7$) and if we set $C=B_1$, then $C^*=A_1$. It then follows that $Aut(\mathcal{T}) = \Z_3$. 

We finally consider the case where a group $G \in \mathcal{G}$ has even order and prove the following:

\begin{thm} \label{them2}
If $\mathcal{H}$ is an HCS($2n+1$) whose full isomorphism group has even order, then either $Aut(\mathcal{H})$ is binary or $2n+1$ is a prime and  $Aut(\mathcal{H})$ is the affine linear group AGL($1,2n+1$).
\end{thm}
\begin{proof} Let $\mathcal{H}$ be an HCS($2n+1$) and assume that $Aut(\mathcal{H})$ has even order.
We first show that any involution of $Aut(\mathcal{H})$ has exactly one fixed point and fixes each cycle of $\mathcal{H}$. This means that an involution is uniquely determined once the point it fixes is known. Therefore, distinct involutory automorphisms of $\mathcal{H}$ have distinct fixed points.

Suppose $\alpha$ is an involutory automorphism of $\mathcal{H}$ and let $x$ be any point not fixed by $\alpha$. Then the edge $[x,\alpha(x)]$ occurs in some cycle $C$, and $\alpha$ fixes this edge. It then follows that the entire cycle $C$ is fixed by $\alpha$ which therefore acts on $C$ as the reflection in the axis of $[x, \alpha(x)]$. The point $a$ opposite to this edge is then the only fixed point of $\alpha$. Note that there are $n$ edges of the form $[x, \alpha(x)]$ with $x\neq a$, and that they are partitioned among the $n$ cycles of $\mathcal{H}$. Therefore, reasoning as before, we get that $\alpha$ fixes all cycles of $\mathcal{H}$

Now, assume that $Aut(\mathcal{H})$ is not binary and let $\beta$ be a second involutory automorphism with fixed point $b$. We denote by $S$ the subgroup of $Aut(\mathcal{H})$ generated by $\alpha$ and $\beta$. Since both involutions fix all cycles of $\mathcal{H}$, all automorphisms in $S$ fix them.
If $C'$ is the cycle of $\mathcal{H}$ containing the edge $[a,b]$, then the map $\beta\alpha$ is the rotation of $C'$ with step $2$. Since $C'$ has odd order, this means that $S$ also contains all rotations of $C'$. Hence, $S$ is the dihedral group of order $4n+2$.

We also have that $2n+1$ is prime. In fact, let $\varphi\in S$ be an automorphism of prime order $p$; given a point $x$ not fixed by $\varphi$ we denote by $C''$ the cycle of $\mathcal{H}$ containing $[x,\varphi(x)]$.
Of course, all edges of the form $[\varphi^{i}(x), \varphi^{i+1}(x)]$ with $i=0,\ldots,p-1$ lies in $C''$.
It follows that $C''$ is the cycle $(x, \varphi(x),$ $\varphi^2(x), \ldots, \varphi^{p-2}(x), \varphi^{p-1}(x))$, hence $2n+1=p$.

One can easily see that $\mathcal{H}$ is then the unique $2$-transitive HCS($p$) whose full automorphism group is AGL($1,p$) \cite{BBM} and this completes the proof.
\end{proof}

The following theorem provides a sufficient condition for a binary group $H$ of order $4m$ to be the full automorphism group of infinitely many HCSs of odd order.

\begin{thm}\label{them3} Let $H$ be a binary group of order $4m$ and let $d\geq3$ be an odd integer.
If there exists a $1$-rotational HCS($4md+1$) 
under $H \times \Z_d$, then there exists an HCS($4md+1$) whose full automorphism group is $H$.
\end{thm}
\begin{proof} Let $\Z_d =\langle z\rangle$ denote the cyclic group of order $d$ generated by $z$ and let $\lambda$ be the unique element of order $2$ in $H$. 
It is straightforward that $G = H \times \Z_d$ is a binary group and its element of order $2$ is $\lambda$. 

Now, let  $\HH$  be a $1$--rotational HCS($4md+1$) under $G$, and let 
$A=(\infty, a_1,$ $a_2, \ldots, a_{2md},  
a_{2md}\lambda,\ldots, a_2\lambda, a_1\lambda)$ 
denote its starter cycle. Also, let 
$Orb_{\Z_d}(A)=\{A_0, A_1, \ldots, A_{d-1}\}$ be the $\Z_d$--orbit of $A$
where $A_i = A \cdot z^i$ for $i=0,1,\ldots,d-1$. We can then see $\HH$ as the union of the 
$H$-orbits of each cycle in $Orb_{\Z_d}(A)$,
that is, 
$\HH= \cup_{i=0}^{d-1} (Orb_{H}(A_i))$. 

We now construct a new HCS($4md+1$)
$\HH^*$ through a slight modification of the cycle $A_0$: 
\begin{enumerate}
  \item  
  Since the binary group $H$ has order $4m$, there is at least an element $x\in H$ of order $4$. 
  Also $\Delta A$ covers all non-zero elements of $H \times \Z_d$. It follows that
  there exists $j\in [2md-1]$ such that $x=a_{j+1}a_{j}^{-1}$. Now, let
  $A_0^*$ denote the graph we get from $A_0$ by replacing the edges in 
  $\EE=\{[a_j,a_{j+1}], [a_{j}\lambda, a_{j+1}\lambda]\}$ with those in
  $\EE^*=\{[a_j,a_{j+1}\lambda], [a_j\lambda, a_{j+1}]\}$.
  \item Set $\HH^* =  \cup_{i=1}^{d-1} Orb_H(A_i) \ \cup \ Orb_H(A_0^*)$.
\end{enumerate}
It is easy to see that $A_0^*$ is a ($4md+1$)--cycle. Also, $E(A_0)\setminus \EE = E(A^*_0)\setminus \EE^*$. We show that $\EE^* = \EE \cdot y$, 
where $y=a_j^{-1}xa_j$. We first point out that $x^2=\lambda$, since $x^2$ has order $2$ and $H$ is binary. Also, note that $a_{j+1} = xa_{j}$ and recall that $\lambda$ commutes with every element in $G$. Therefore,
\begin{align*}
& [a_j, a_{j+1}]y= [xa_{j}, x^2a_j] = [a_{j+1}, a_j\lambda], \;\;\; \text{and} \\
& [a_j\lambda, a_{j+1}\lambda]y= [a_j, a_{j+1}]y\lambda = [a_{j+1}, a_j\lambda]\lambda = 
[a_{j+1}\lambda, a_j]. 
\end{align*}
Since $y \in H$, we have that $Orb_{H}(\EE) = Orb_{H}(\EE^*)$. It follows that $Orb_H(A_0)$ and 
$Orb_H(A_0^*)$ cover the same set of edges. 
Since $\HH\setminus Orb_H(A_0) = \HH^*\setminus Orb_H(A^*_0)$, we have that $\HH^*$ covers the same set of edges covered by $\HH$, that is, $\HH^*$ is an HCS($4md+1$).  
  
We now show that the full automorphism groups of $\HH^*$ is isomorphic to $H$.  
We set $\AAA = Aut(\HH^*)$ and denote by $\AAA_{\infty}$ the $\AAA$-stabilizer of $\infty$.
Also, let $\tau_g$ denote the translation by the element $g\in G$, that is, the permutation on 
$G \ \cup \ \{\infty\}$ fixing $\infty$ and mapping $x \in G$ to $xg$ for any $x,g\in G$.
Finally, let $\tau_G$ and $\tau_H$ denote the group of all translations by the elements of $G$ and $H$, respectively. It is easy to check that, by construction, $\tau_H$ is an automorphism group of $\HH^*$ fixing $\infty$, that is, $\tau_{H} \subseteq \AAA_{\infty}$; on the other hand, the replacement of $Orb_H(A_0)$ with $Orb_H(A^*_0)$ ensures that $\tau_g$ is not an automorphism of $\HH^*$ whenever $g \in G\setminus H$. In other words, for any $g\in G$ we have that
\begin{equation}\label{translations}
  \tau_{g} \in \AAA \;\;\; \text{if and only if} \;\;\; g \in H.
\end{equation}

We are going to show that 
$\AAA_{\infty}=\tau_H$. Let $\varphi \in \AAA_{\infty}$ and note that 
$|Orb_H(A_0^*)|<|\HH^*\setminus Orb_H(A_0^*)|$. Therefore, there exists a cycle 
$C=(\infty, g_1, g_2, \ldots, g_{4mp})\in \HH^*\setminus Orb_H(A_0^*)$ such that 
$\varphi(C)=\HH^*\setminus Orb_H(A_0^*)$. Since $\varphi$ fixes $\infty$, then
$\varphi(C) = (\infty, \varphi(g_1),\varphi(g_2), \ldots, \varphi(g_{4mp}))$. 
Note that $\HH^*\setminus Orb_H(A_0^*) \subseteq \HH$, 
therefore $\varphi(C)$ is a translate of $C$, namely, there exists 
$x\in G$ such that $\varphi(g_i) = g_i\cdot x$. This means that $\varphi = \tau_x$, and in view of \eqref{translations}, $x \in H$. It follows that $\varphi \in \tau_H$, hence 
$\AAA_{\infty} = \tau_{H}$.
  
 Since $\tau_H$ is isomorphic to $H$,  we are left to show $\AAA=\AAA_{\infty}$.
 Assume that there exists $w \in Orb_{\AAA}(\infty) \setminus \{\infty\}$. From a classic result on permutations groups, we have that $\AAA_{w}$ and $\AAA_{\infty}$ are conjugate hence, in particular, they are isomorphic. This means that $\AAA_w$ contains an involution $\psi$.
  Since $\AAA_{\infty} \ \cap \ \AAA_{w} = \{id\}$ (otherwise there would be a non-trivial automorphism fixing two vertices), we have that $\tau_{\lambda}$ and $\psi$ are distinct. Therefore, in view of Theorem \ref{them2}, we have that $4md+1$ is a prime and $\AAA = AGL(1,4md+1)$; in particular, $|\AAA| = 4md(4md+1)$. 
%  But this is a contradiction since 
%  $|\AAA| = 4m(4md+1)$ and $d \geq 3$. Therefore, $Orb_{\AAA}(\infty) = \{\infty\}$, namely, 
%  $\AAA=\AAA_{\infty}$.
   But this leads to a contradiction since 
  $|\AAA| = |\AAA_{\infty}||Orb_{\AAA}(\infty)|$ and $|Orb_{\AAA}(\infty)|\leq 4md+1$, that is, $|\AAA| \leq 4m(4md+1)$ where $d \geq 3$. 
  Therefore, $Orb_{\AAA}(\infty) = \{\infty\}$, namely,  $\AAA=\AAA_{\infty}$.
\end{proof}

As a consequence we obtain the following: 

\begin{thm} \label{cor}
The quaternion group $\Q_8$ is the full automorphism group of an HCS($8d+1$), for any $d \geq 3$.
\end{thm}
\begin{proof} 
It is enough to observe that $\Q_{8}\times \Z_{d}$, $d\geq 3$, is a binary solvable group 
$\neq \Q_8$. As mentioned earlier in this paper, any binary solvable group is the full automorphism group of an HCS of odd order. Then, the conclusion immediately follows by Theorem \ref{them3}
\end{proof}

Collecting the above results, we can prove Theorem \ref{them0}.

\begin{proof}[proof of Theorem \ref{them0}]
The first part of the statement is proven in Theorem \ref{them2}. 

Now, let $G$ be a finite group. If $G$ has odd order or it is $\Q_8$, then by Theorems \ref{them1} and \ref{them3} we have that $G$ is the full automorphism group of a suitable HCS of odd order. If $G$ is a binary solvable group $\neq \Q_8$ or AGL(1,p) with $p$ prime, it is known \cite{BRT, BBM} that $G \in \mathcal{G}$. 
\end{proof}

\begin{rem} 
  Note that Theorem \ref{them3} would allow us to solve completely the problem under investigation in this paper, if one could show that any sufficiently large binary group  
has a $1$--rotational action on an HCS of odd order. In this case, given a binary non-solvable group $H$ of order $4m$, we could consider the group $G = H \times \Z_d$ with $d$ sufficiently large to ensure that $G$ has a $1$--rotational action on an HCS($4md+1$). By Theorem \ref{them3} we would have that $H$ is the full automorphism group of an HCS of odd order.
\end{rem}

\end{document}